\documentclass[12pt]{amsart}
\usepackage{amsmath,amssymb,amsthm,amsfonts,stmaryrd,graphicx}
\usepackage{stmaryrd}
\usepackage[english]{babel}
\usepackage{caption}
\newtheorem{theorem}{{\bf Theorem}}[section]

\newtheorem{proposition}{{\bf Proposition}}[section]

\theoremstyle{definition}
\newtheorem{definition}{{\bf Definition}}[section]
\theoremstyle{remark}

\numberwithin{equation}{section}

\newcommand{\N}{\mathbb{N}}

\newcommand{\qbinomial}[3]{\mbox{$
\biggl[ 
\begin{array}{c}
#1\\
 #2
\end{array}\biggr]_{
\!{#3}} $} }

\newcommand{\be}{\begin{equation}}
\newcommand{\ee}{\end{equation}}
\newcommand{\bea}{\begin{eqnarray}}
\newcommand{\eea}{\end{eqnarray}}
\newcommand{\bd}{\begin{displaymath}}
\newcommand{\ed}{\end{displaymath}}

\begin{document}

\title[ Some homogeneous $q$-difference operators and the associated  generalized Hahn polynomials
 ]{Some homogeneous $q$-difference operators and the associated  generalized Hahn polynomials }%
\author{Hari M. Srivastava$^{1,\ddag}$,  Sama Arjika$^{1,\ast}$ and Abey Kelil$^{3,\dag}$}%
\address{$^{1}$Department of Mathematics and Statistics, University of Victoria,
Victoria, British Columbia V8W 3R4, Canada}
\address{${}^{2}$Faculty of Sciences and Technics,   University of Agadez,   Niger}
\address{$^{3}$Department of Mathematics and Applied Mathematics, University of Pretoria,   S. A.}

\email{ $\ddag$harimsri@math.uvic.ca,$\ast$rjksama2008@gmail.com, $\dag$ abeysh2001@gmail.com}%

\begin{abstract}
In this paper,   we first construct the homogeneous $q$-shift operator $\widetilde{E}(a,b;D_{q})$  and the homogeneous $q$-difference operator $\widetilde{L}(a,b; \theta_{xy})$.  We then apply these operators in order to represent
and investigate generalized Cauchy and a general form of Hahn polynomials. 
We derive some $q$-identities such as:  generating functions,   extended generating functions,  Mehler's formula and   Roger's formula for these $q$-polynomials.  
\end{abstract}

\maketitle

\section{Introduction}
We adopt the common conventions and notations on $q$-series.  For the convenience of the reader, we provide a  summary of the mathematical notations and definitions to be used  in this paper. We refer  to the general references (see \cite{Koekock})  for the definitions and notations. Throughout this paper, we assume that  $0<q \leq 1$. 

For complex  numbers $a$, the $q$-shifted factorials are defined by:
\be
(a;q)_{0}:=1,\; (a;q)_{n}:=\prod_{k=0}^{n-1}(1-aq^{k}), \;\;n\in\N
\ee 
and for large $n$, we have  
\be
 (a;q)_{\infty}:=\prod_{k=0}^{\infty}(1-aq^{k}).
\ee
The $q$-numbers and  $q$-factorials are defined as follows:
\be 
 [n]_{q}:=\frac{1-q^n}{1-q},\quad  [n]_{q}!:=\prod_{k=1}^n\frac{1-q^k}{1-q},\quad [0]_{q}!:=1.
\ee
The $q$-binomial coefficients   are given by  
\bea 
\label{samz}
 {\,n\,\atopwithdelims []\,k\,}_{q} : &=&\frac{[n]_{q}! }
{[k]_{q}!\,[n-k]_{q}!}.
\eea
The    basic or q-hypergeometric function 
  in the variable $x$ (see  Slater \cite[Chap. 3]{SLATER},  Srivastava and Karlsson   \cite[p.347, Eq. (272)]{SrivastaKarlsson} 
 for details) is defined as \cite{Koekock}:
\be 
\label{defe}
{}_{r}\Phi_s\left(\begin{array}{c}a_1, a_2,\ldots, a_r
 \\
b_1,b_2,\ldots,b_s
 \end{array}\Big|
q;x\right)
 =\sum_{n=0}^\infty\Big[(-1)^n q^{({}^n_2)}\Big]^{1+s-r}\,\frac{(a_1, a_2,\ldots, a_r;q)_n}{(b_1,b_2,\ldots,b_s;q)_n}\frac{ x^n}{(q;q)_n}.  
\ee
Basic  or $q$-hypergeometric series and various associated families of $q$-polynomials are useful in a wide variety of fields including, for example, the theory of partitions, number theory, combinatorial analysis, finite vector spaces, Lie theory, particle physics, non-linear electric circuit theory, mechanical engineering, theory of heat conduction, quantum mechanics, cosmology, and statistics (see \cite[pp. 346-351]{SrivastaKarlsson} and the references cited therein).

We will be mainly concerned with the Cauchy polynomials as given
below  \cite{GasparRahman}
\bea
\label{def}
 p_n(x,y):=(x-y)( x- qy)\cdots ( x-q^{n-1}y) =(y/x;q)_nx^n
\eea
with the generating function \cite{Chen2003}
\be
\label{gener}
\sum_{k=0}^{\infty} p_n(x,y)
\frac{t^n }{(q;q)_k} = 
\frac{(yt;q)_\infty}{(xt;q)_\infty},
\ee
where \cite{Chen2003}
\be
\label{asq}
  p_n(x,y)=(-1)^n q^{({}^n_2)}p_n(y,q^{1-n}x),
\ee
and
\be
\label{aasq}
  p_{n-k}(x,q^{1-n}y)=(-1)^{n-k} q^{({}^k_2)-({}^n_2)}p_{n-k}(y,q^{k}x).
\ee
 which naturally arise in the $q$-umbral calculus  \cite{Andrews}, Goldman and Rota \cite{Roman1982},  Ihrig and Ismail \cite{Ismail1981}, Johnson \cite{Johnson} and Roman \cite{Roman1982}. The generating function (\ref{gener}) is also the homogeneous version
of the Cauchy identity or the $q$-binomial theorem  \cite{GasparRahman}
\be
\label{putt}
\sum_{k=0}^{\infty} 
\frac{(a;q)_k }{(q;q)_k}x^{k}={}_{1}\Phi_0\left(\begin{array}{c}a
 \\
-
 \end{array}\Big|
q;x\right)= 
\frac{(ax;q)_\infty}{(x;q)_\infty}. 
\ee
Putting    $a=0$, the relation (\ref{putt}) becomes  Euler's identity  \cite{GasparRahman}
\be
\label{q-expo-alpha}
  \sum_{k=0}^{\infty} \frac{ x^{k}}{(q;q)_k}=\frac{1}{(x;q)_\infty} 
\ee
and its inverse relation  \cite{GasparRahman}
\be
\label{q-Expo-alpha}
 \sum_{k=0}^{\infty} 
\frac{(-1)^kq^{ ({}^k_2)
}\,x^{k}}{(q;q)_k}=(x;q)_\infty.
\ee

 Recently, Saad and Sukhi  \cite{Saad}  have introduced the following q-exponential operator $R(b{D}_q)$   as follows 
\be 
\label{mats}
R(bD_q)=\sum_{k=0}^\infty \frac{(-1)^kq^{({}^k_2)}  }{(q;q)_k}\, \left(b\,{D}_q\right)^k
\ee 
where the $q$-differential operator, or the $q$-derivative, acting on the variable a, is defined by \cite{Liu97,Saadsukhi} 
\be
\label{deffd}
 {D}_q\big\{f(a)\big\}=\frac{f(a)-f( q a)}{ a}.
\ee
Evidently,
\be 
R(yD_q)(x^n)=p_n(x,y).
\ee
Suppose that the operator $R(yD_q)$ acts on the variable $x$, then we have the following  \cite{Saad,SrivastaAbdlhusein}:
\be 
R(yD_q)
\Bigg\{\frac{1}{(xt;q)_\infty}\Bigg\} = \frac{(yt;q)_\infty }{(xt;q)_\infty},
\ee
\be 
R(yD_q)
\Bigg\{\frac{1 }{(xt,xs;q)_\infty}\Bigg\} =\frac{(ys;q)_\infty }{(xt,xs;q)_\infty}{}_{1}\Phi_1\left(\begin{array}{c}xs
 \\
ys
 \end{array}\Big|
q;yt\right),\label{v09}
\ee
\be 
R(yD_q)
\Bigg\{\frac{(xv;q)_\infty }{(xt,xs;q)_\infty}\Bigg\} =\frac{(ys;q)_\infty }{(xs;q)_\infty}{}_{2}\Phi_1\left(\begin{array}{c}v/t, y/x
 \\
ys
 \end{array}\Big|
q;xt\right). \label{v0}
\ee
 Srivastava and Abdlhusein  \cite{SrivastaAbdlhusein}, showed that by setting $v = 0$ in (\ref{v0}), it becomes:  
\be 
R(yD_q)
\Bigg\{\frac{1 }{(xt,xs;q)_\infty}\Bigg\} =\frac{(ys;q)_\infty }{(xs;q)_\infty}{}_{2}\Phi_1\left(\begin{array}{c}y/x,0
 \\
ys
 \end{array}\Big|
q;xt\right).  \label{v10}
\ee
Comparing the identity (\ref{v10}) and the identity (\ref{v09}), we get the following transformation
\be 
\label{sub}
 {}_{2}\Phi_1\left(\begin{array}{c}y/x,0
 \\
ys
 \end{array}\Big|
q;xt\right)=\frac{1 }{(xt;q)_\infty}{}_{1}\Phi_1\left(\begin{array}{c}xs
 \\
ys
 \end{array}\Big|
q;yt\right).
\ee

Motivated by Saad and Sukhi  \cite{Saad},   Srivastava and Abdlhusein   \cite{SrivastaAbdlhusein} works,   our interest    is to  introduce new  homogeneous $q$-difference operators    $\widetilde{E}(a,b;D_{q})$ and $\widetilde{T}(a,b;\theta_{xy})$.  The first  homogeneous q-difference operator 
$\widetilde{E}(a,b;D_{q})$ is defined 
 by
\be 
\label{samat}
\widetilde{E}(a,b; {D}_{q})=\sum_{k=0}^\infty \frac{(-1)^kq^{({}^k_2)}\,(a;q)_k }{(q;q)_k}\, \left(b\,{D}_{q}\right)^k.
\ee
Compared with $R(yD_q)$, the homogeneous q-difference operator  (\ref{samat}) involves two parameters. Clearly, the operator $R(yD_q)$ can be considered as a special case of the   operator (\ref{samat}) for  $a = 0$.

Second, we introduce another homogeneous q-difference operator  $\widetilde{L}(a,b; \theta_{xy})$ by
\be 
\label{aqamat}
\widetilde{L}(a,b; \theta_{xy})=\sum_{k=0}^\infty \frac{q^{({}^k_2)}\,(a;q)_k }{(q;q)_k}\, \left(b\,\theta_{xy}\right)^k,
\ee
where \cite{Liu97,Saadsukhi} 
\be
\label{deffd}
  \theta_{xy}\Big\{f(x,y)\Big\}=\frac{f(q^{-1}x,y)-f(  x,qy)}{ q^{-1}x- y}
\ee
 acting on  functions of suitable  variables $x$ and $y$  \cite{SrivastaAbdlhusein,Saadsukhi}
\be 
\label{deule2}
 \theta_{xy}^k\, p_n(y,x)=(-1)^k\frac{(q;q)_n}{(q;q)_{n-k}}p_{n-k}(y,x)\mbox{ and }  \theta_{xy}^k\Bigg\{\frac{(xt;q)_\infty}{(yt;q)_\infty}\Bigg\} =(-t)^k\frac{(xt;q)_\infty}{(yt;q)_\infty}.
\ee
The homogeneous q-difference operator  (\ref{aqamat}) involves two parameters. It can   be considered as a generalization of the  homogeneous $q$-difference operator  $ {L}(b \theta_{xy})$  introduced by Saad and Sukhi \cite{Saadsukhi}. For  $a = 0$, we have: 
\be 
\widetilde{L}(0,b; \theta_{xy}) = {L}(b \theta_{xy}).
\ee 
The    definition (\ref{aqamat}) is motivated by the natural question of extending generating function, Mehler's formula and Roger's-type formula for a general form for the Hahn polynomials $ h_n(x,y,a,b|q)$.  The    operators (\ref{samat}) and (\ref{aqamat}) turn out to be suitable for dealing with  a generalized Cauchy polynomials $p_n(x,y,a)$  and   the generalized Hahn polynomials $h_n(x,y,a,b|q)$. They are  then applied in order to represent and investigate   some $q$-identities such as:  generating functions,   extended generating functions,  Mehler's formula and   Roger's-type formula for $p_n(x,y,a)$  and    $h_n(x,y,a,b|q)$ polynomials. 
 
\section{Generalized   Cauchy polynomials $ p_n(x,y,a)$}
In this section, we introduce a generalized   Cauchy polynomials $ p_n(x,y,a)$.
We then represent this   polynomials $ p_n(x,y,a)$ by means of the homogeneous
 $q$-difference operator    $\widetilde{E}(a,b;D_{q})$ and  derive their generating function.
We also use the operational formula for $ p_n(x,y,a)$
  to establish an extended generating function,  Mehler’s formula and   Rogers formula for the   Cauchy polynomials $ p_n(x,y,a)$. 

Let us start by the   the following Leibniz rule \cite{Roman1985}
\be
\label{Lieb}
D_q^n\left\{ f(x)g(x)\right\}=\sum_{k=0}^n \qbinomial{n}{k}{q} q^{k(k-n)} D_q^k\left\{f(x)\right\} D_q^{n-k}\left\{g(q^kx)\right\}.
\ee 
For $f(x)=x^k,\,  f(x)=(x t;q)_\infty$ and $\displaystyle f(x)= {}_{1}\Phi_1\left(\begin{array}{c}a
 \\
0
 \end{array}\Big|
q;bx\right),$ respectively,  we have the following identities:
\be
\label{rLieb}
D_q^n (x^k)= (q^{k-n+1};q)_n x^{k-n},\quad 
D_q^n(x t;q)_\infty = q^{({}^n_2)}t^n\frac{(x t;q)_\infty }{(x t;q)_n}
\ee
and
\be
\label{assertion}
D_q^n  {}_{1}\Phi_1\left(\begin{array}{c}a
 \\
0
 \end{array}\Big|
q;bx\right)= (-1)^nb^n (a;q)_nq^{({}^n_2)}{}_{1}\Phi_1\left(\begin{array}{c}aq^n
 \\
0
 \end{array}\Big|
q;bxq^n\right).
\ee
Suppose that the operator $ D_q$ acts on the variable $s$. If we set $f(s)=s^n,\,g(s)=s^m$ and  $\displaystyle f(s)={}_{1}\Phi_1\left(\begin{array}{c}a
 \\
0
 \end{array}\Big|
q;bsq^k\right),\,\, g(s)=\frac{1 }{(xsq^k;q)_\infty},$ respectively, and making used  (\ref{rLieb}) and (\ref{Lieb}), we get  the following identities to be used in the sequel: 
\be 
\Big( q^{n+m-k+1};q\Big)_k=
\sum_{j=0}^k \qbinomial{k}{j}{q}  q^{j(j-k+m)}  (q^{n -j+1};q)_j (q^{m -k+j+1};q)_{k-j}
\ee 
and
\be
 D_q^{n} \left\{\frac{{}_{1}\Phi_1\left(\begin{array}{c}a
 \\
0
 \end{array}\Big|
q;bsq^k\right)}{(xsq^k;q)_\infty} \right\}\label{ase}
\ee
\bea
 =\frac{  1}{(xs;q)_\infty} \sum_{j=0}^n \qbinomial{n}{j}{q} q^{j(j-n+k)+({}^j_2)} x^j  
  (xs;q)_jD_q^{n-j} {}_{1}\Phi_1\left(\begin{array}{c}a
 \\
0
 \end{array}\Big|
q;bsq^{j+k}\right). \nonumber
\eea 

 \begin{definition}
The  homogeneous $q$-difference operator 
  $\widetilde{E}(a,b;D_{q})$  is defined by
\be 
\label{aamat}
\widetilde{E}(a,b; {D}_{q})=\sum_{k=0}^\infty \frac{(-1)^kq^{({}^k_2)}\,(a;q)_k }{(q;q)_k}\, \left(b\,{D}_{q}\right)^k.
\ee
\end{definition}
Remark that, for $a=0,$ we have: 
\be
\widetilde{E}(0,b; {D}_{q})=R(bD_q).
\ee
\begin{proposition}
We suppose that the operator $ {D}_{q}$ acts on the variable $x$. We have:
\be 
\label{1A}
\widetilde{E}(a,y; {D}_{q})
\Bigg\{\frac{1}{(xt;q)_\infty}\Bigg\} =  \frac{1}{(xt;q)_\infty} \;  {}_{1}\Phi_1\left(\begin{array}{c}a
 \\
0
 \end{array}\Big|
q;yt\right),
\ee
\be 
\label{1A1}
\widetilde{E}(a,y; {D}_{q})
\Bigg\{\frac{1 }{(xt,xs;q)_\infty}\Bigg\} =
\ee
\be
\frac{1}{(xt,xs;q)_\infty}\sum_{n=0}^\infty \frac{(-1)^nq^{({}^n_2)}\,(a;q)_n(sy)^n }{(q;q)_n}\, {}_{2}\Phi_1\left(\begin{array}{c}  aq^n, xs
 \\
0
 \end{array}\Big|
q; ytq^{n}\right), \nonumber 
\ee
\be 
\label{1A2}
\widetilde{E}(a,y; {D}_{q})\Bigg\{\frac{x^k}{(xt;q)_\infty}\Bigg\}
=
\ee
\be
 \frac{x^k}{(xt;q)_\infty}\sum_{n=0}^\infty \frac{(-1)^nq^{({}^n_2)}\,(a;q)_n(ty)^n }{(q;q)_n}\, {}_{3}\Phi_2\left(\begin{array}{c}q^{-k}, aq^n, xt
 \\
0,0
 \end{array}\Big|
q;\frac{y}{x}q^{n+k}\right). \nonumber 
\ee 
\end{proposition}
 Setting $k=0$ in the assertion   (\ref{1A2}) and making used  $\displaystyle {}_{3}\Phi_2\left(\begin{array}{c}1, aq^n, xt
 \\
0,0
 \end{array}\Big|
q;\frac{y}{x}q^{n}\right)=1,$ we get  (\ref{1A}). 
\begin{theorem}
\label{rede}
We suppose that the operator $ {D}_{q}$ acts on the variable $s$. We have:
\be  
\label{ddredes}
\widetilde{E}(a,t; {D}_{q})\Bigg\{\frac{(ys;q)_\infty}{(xs;q)_\infty} \;  {}_{1}\Phi_1\left(\begin{array}{c}a
 \\
0
 \end{array}\Big|
q;bs\right)\Bigg\}
=\frac{ (ys;q)_\infty}{(xs;q)_\infty}   \sum_{n=0}^\infty\sum_{k=0}^\infty \sum_{j=0}^\infty\frac{  (-1)^{j+k}q^{ ({}^k_2)+({}^n_2)+({}^{n+j+k}_{\,\,\,\,\,\,2})}   }{  (q;q)_k(q;q)_{n}(q;q)_j}  \nonumber
\ee  
\be
 \times\,\frac{  (a;q)_n(a;q)_{n+j+k}  (xs;q)_k    }{ (ys;q)_k } 
  {}_{1}\Phi_1\left(\begin{array}{c}aq^n
 \\
0
 \end{array}\Big|
q;bsq^{j+k+n}\right) \,t^{n+j+k}.
\ee 
\end{theorem}
\begin{proof}
\be
\widetilde{E}(a,t; {D}_{q})\Bigg\{\frac{(ys;q)_\infty}{(xs;q)_\infty} \;  {}_{1}\Phi_1\left(\begin{array}{c}a
 \\
0
 \end{array}\Big|
q;bs\right)\Bigg\}\nonumber
\ee
\be
=\sum_{n=0}^\infty \frac{(-1)^nq^{({}^n_2)}\,t^n(a;q)_n }{(q;q)_n}\,  {D}_{q}^n
\Bigg\{\frac{(ys;q)_\infty}{(xs;q)_\infty} \;  {}_{1}\Phi_1\left(\begin{array}{c}a
 \\
0
 \end{array}\Big|
q;bs\right)\Bigg\}. \label{labs}
\ee 
By using (\ref{Lieb}),  the r.h.s. of (\ref{labs})  becomes
\be
\sum_{n=0}^\infty\frac{(-1)^nq^{({}^n_2)}\,t^n(a;q)_n}{(q;q)_n}  \sum_{k=0}^n \qbinomial{n}{k}{q} q^{k(k-n)}D_q^k(ys;q)_\infty
 D_q^{n-k} \left\{\frac{{}_{1}\Phi_1\left(\begin{array}{c}a
 \\
0
 \end{array}\Big|
q;bsq^k\right)}{(xsq^k;q)_\infty} \right\}\nonumber
\ee 
\be
= \sum_{n=0}^\infty\sum_{k=n}^\infty\frac{ (-1)^nq^{({}^n_2)+k(k-n)+({}^k_2)}\,t^n\,   y^k(a;q)_n }{(q;q)_k(q;q)_{n-k}} 
 (ysq^k;q)_\infty
 D_q^{n-k} \left\{\frac{{}_{1}\Phi_1\left(\begin{array}{c}a
 \\
0
 \end{array}\Big|
q;bsq^k\right)}{(xsq^k;q)_\infty} \right\}. \nonumber
\ee
Changing $n-k$ by $n$ in the last relation,  it becomes
\be
(ys;q)_\infty   \sum_{n=0}^\infty\sum_{k=0}^\infty\frac{  (-1)^{n+k}q^{({}^{n+k}_{\,\,\,2})-kn+({}^k_2)}\,t^{n+k}\,   y^k(a;q)_{n+k}   }{ (ys;q)_k(q;q)_k(q;q)_{n}} 
 D_q^{n} \left\{\frac{{}_{1}\Phi_1\left(\begin{array}{c}a
 \\
0
 \end{array}\Big|
q;bsq^k\right)}{(xsq^k;q)_\infty} \right\}.
  \label{aseeq}
\ee 
By using (\ref{ase}), the last relation takes the form
\be
 \frac{ (ys;q)_\infty}{(xs;q)_\infty}   \sum_{n=0}^\infty\sum_{k=0}^\infty \sum_{j=0}^n\frac{  (-1)^{n+k}q^{-kn+j(j-n+k)+({}^k_2)+({}^{n+k}_{\,\,\,2})} (a;q)_{n+k}  (xs;q)_k  y^k t^{n+k}  x^j  }{ (ys;q)_k(q;q)_k(q;q)_{n-j}(q;q)_j} 
\nonumber
\ee 
\be
\times
 D_q^{n-j} {}_{1}\Phi_1\left(\begin{array}{c}a
 \\
0
 \end{array}\Big|
q;bsq^{j+k}\right)\nonumber
\ee
\be
 =\frac{ (ys;q)_\infty}{(xs;q)_\infty}  \sum_{n=0}^\infty\sum_{k=0}^\infty \sum_{j=0}^\infty\frac{  (-1)^{n+j+k}q^{-kn-jn+({}^k_2)+({}^{n+j+k}_{\,\,\,\,\,\,2})} (a;q)_{n+j+k}  (xs;q)_k  y^k t^{n+j+k}  x^j  }{ (ys;q)_k(q;q)_k(q;q)_{n}(q;q)_j} 
\nonumber
\ee 
\be
\times\,
 D_q^{n} {}_{1}\Phi_1\left(\begin{array}{c}a
 \\
0
 \end{array}\Big|
q;bsq^{j+k}\right)\nonumber
\ee 
\be
 =\frac{ (ys;q)_\infty}{(xs;q)_\infty}   \sum_{n=0}^\infty\sum_{k=0}^\infty \sum_{j=0}^\infty\frac{  (-1)^{j+k}q^{ ({}^k_2)+({}^n_2)+({}^{n+j+k}_{\,\,\,\,\,\,2})} (a;q)_n(a;q)_{n+j+k}  (xs;q)_k   b^n y^k t^{n+j+k}  x^j  }{ (ys;q)_k(q;q)_k(q;q)_{n}(q;q)_j} 
\nonumber
\ee 
\be
\times\,
  {}_{1}\Phi_1\left(\begin{array}{c}aq^n
 \\
0
 \end{array}\Big|
q;bsq^{j+k+n}\right)  \nonumber
\ee 

Summarizing the above calculations,   we get the assertion (\ref{ddredes}). 
\end{proof}

\begin{definition}
In terms of  $q$-shifted factorial, we define a generalized      Cauchy polynomials $p_n(x,y,a)$ by
\be
\label{defR}
p_n(x,y,a)=\sum_{k=0}^n  \qbinomial{n}{k}{q}(-1)^kq^{\binom{k}{2}}\,(a;q)_{k}x^{n-k} y^k 
\ee
and 
\bea
\label{defRd}
p_n(0,y,a)= (-1)^nq^{\binom{n}{2}}\,(a;q)_{n}  y^n,\,  p_n(x,0,a)= x^{n },\, p_n(x,y,0)=\sum_{k=0}^n  \qbinomial{n}{k}{q}(-1)^kq^{\binom{k}{2}} x^{n-k} y^k.\nonumber
\eea
\end{definition}
The following operational formula holds true.
 \be
\label{sddefR}
\widetilde{E}(a,y; {D}_q)\Big\{x^n\Big\}=p_n(x,y,a).
\ee
Indeed, 
 \bea
\widetilde{E}(a,y; {D}_q)\Big\{x^n\Big\}&=& \sum_{k=0}^\infty \frac{(-1)^kq^{({}^k_2)}\,(a;q)_k }{(q;q)_k}\, \left(y\,{D}_{q}\right)^k\Big\{x^n\Big\}\cr
&=&\sum_{k=0}^\infty \frac{(-1)^kq^{({}^k_2)}\,(a;q)_k }{(q;q)_k}\frac{ (q;q)_n }{(q;q)_{n-k}}\, y^k x^{n-k} =p_n(x,y,a).
\eea
\begin{theorem}(Generating function for $p_n(x,y,a)$)
\label{csobbvrro}
The following generating function holds true.
\be
\label{aaGlnene}
 \sum_{n=0}^\infty p_n(x,y,a) \frac{t^n}{(q;q)_n}= \frac{1}{(xt;q)_\infty} \;  {}_{1}\Phi_1\left(\begin{array}{c}a
 \\
0
 \end{array}\Big|
q;yt\right),
\ee
where $|xt| <1.$
\end{theorem}
\begin{proof} We suppose that the operator $ {D}_{q}$ acts on the variable $x$.  Then, 
\bea 
 \sum_{n=0}^\infty p_n(x,y,a) \frac{t^n}{(q;q)_n}=\sum_{n=0}^\infty \widetilde{E}(a,y; {D}_{q})\Big\{x^n\Big\}\frac{t^n}{(q;q)_n}&= &\widetilde{E}(a,y; {D}_{q})\Bigg\{\sum_{n=0}^\infty  \frac{(xt)^n}{(q;q)_n}\Bigg\}\cr
 &=&\widetilde{E}(a,y; {D}_{q})\Bigg\{ \frac{1}{(xt;q)_\infty}\Bigg\}\cr
&=& \frac{1}{(xt;q)_\infty} \;  {}_{1}\Phi_1\left(\begin{array}{c}a
 \\
0
 \end{array}\Big|
q;yt\right)\nonumber
\eea
which evidently completes the proof of assertion (\ref{aaGlnene}). 
\end{proof}


\begin{theorem}(Extended generating function for $p_n(x,y,a)$)
\label{orro}
We have:
\be
\label{ggdene}
 \sum_{n=0}^\infty p_{n+k}(x,y,a) \frac{t^n}{(q;q)_n}
\ee
\be
 \frac{x^k}{(xt;q)_\infty}\sum_{n=0}^\infty \frac{(-1)^nq^{({}^n_2)}\,(a;q)_n(ty)^n }{(q;q)_n} {}_{3}\Phi_2\left(\begin{array}{c}q^{-k}, aq^n, xt
 \\
0,0
 \end{array}\Big|
q;\frac{y}{x}q^{n+k}\right), \nonumber 
\ee 
where $|xt| <1.$
\end{theorem}
\begin{proof} We suppose that the operator $ {D}_{q}$ acts on the variable $s$. Then, we have: 
\bea 
\sum_{n=0}^\infty  p_{n+k}(x,y,a) \frac{t^n}{(q;q)_n}&=&\sum_{n=0}^\infty \widetilde{E}(a,y; {D}_{q})\Big\{x^{n+k}\Big\}\frac{t^n}{(q;q)_n} 
= \widetilde{E}(a,y; {D}_{q})\Bigg\{x^k\sum_{n=0}^\infty  \frac{(xt)^n}{(q;q)_n}\Bigg\} \cr
&=&\widetilde{E}(a,y; {D}_{q})\Bigg\{\frac{x^k}{(xt;q)_\infty}\Bigg\}.\nonumber
\eea
The proof of assertion (\ref{ggdene})  is achieved by using the relation (\ref{1A2}).  
\end{proof}
Setting $k=0$ in  Theorem \ref{orro}, we get the generating function (\ref{aaGlnene}).  

\begin{theorem}(Roger's-Type formula for $p_n(x,y,a)$)
We have:
\be 
\label{ule2}
\sum_{n=0}^\infty\sum_{m=0}^\infty p_{n+m}(x,y,a)    \frac{  t^n}{(q;q)_n}  \frac{  s^m}{(q;q)_m}=
\ee 
\be
\frac{1}{(xt,xs;q)_\infty}\sum_{n=0}^\infty \frac{(-1)^nq^{({}^n_2)}\,(a;q)_n(sy)^n }{(q;q)_n} {}_{2}\Phi_1\left(\begin{array}{c}  aq^n, xs
 \\
0
 \end{array}\Big|
q; ytq^{n}\right) \nonumber 
\ee
where $max\{|xt|, |xs|\} <1.$
\end{theorem}
 \begin{proof} We suppose that the operator $ {D}_{q}$ acts on the variable $x$. We have:
\be 
 \sum_{n=0}^\infty\sum_{m=0}^\infty p_{n+m}(x,y,a)    \frac{  t^n}{(q;q)_n}  \frac{  s^m}{(q;q)_m}
=  
\sum_{n=0}^\infty\sum_{m=0}^\infty \widetilde{E}(a,y; {D}_{q})\Big\{x^{n+m}\Big\}\frac{t^n}{(q;q)_n}  \frac{  s^m}{(q;q)_m}\nonumber \ee
\be 
= \widetilde{E}(a,y; {D}_{q})\Bigg\{\sum_{n=0}^\infty  \frac{(xt)^n}{(q;q)_n}\sum_{n=0}^\infty  \frac{(xs)^m}{(q;q)_m}\Bigg\}
=\widetilde{E}(a,y; {D}_{q})\Bigg\{ \frac{1}{(xt,xs;q)_\infty} \Bigg\}.\nonumber
\ee 
 The proof is achieved by using (\ref{1A1}).
\end{proof}
Now, we aim to present an operator approach to Mehler's formula for the generalized Cauchy polynomials  $p_n(x,y,a)$.
\begin{theorem}(Mehler's  formula for  $p_n(x,y,a)$)
We have:
\be 
\label{mehl}
\sum_{n=0}^\infty p_n(x,y,a) p_n(u,v,\alpha)   \frac{  t^n}{(q;q)_n}= \widetilde{E}(a,y; {D}_{q})\Bigg\{ \frac{1}{(uxt;q)_\infty}    \;  {}_{1}\Phi_1\left(\begin{array}{c}\alpha
 \\
0
 \end{array}\Big|
q;vxt\right) \Bigg\}  .  
\ee 
 
\end{theorem}
\begin{proof}
\be 
 \sum_{n=0}^\infty p_n(x,y,a) p_n(u,v,\alpha)    \frac{  t^n}{(q;q)_n}
=  
\sum_{n=0}^\infty \widetilde{E}(a,y; {D}_{q})\Big\{x^{n}\Big\} p_n(u,v,\alpha)\frac{t^n}{(q;q)_n}= \nonumber \ee
\be 
\widetilde{E}(a,y; {D}_{q})\Bigg\{\sum_{n=0}^\infty    p_n(u,v,\alpha)\frac{(xt)^n}{(q;q)_n}  \Bigg\}
= \widetilde{E}(a,y; {D}_{q})\Bigg\{ \frac{1}{(uxt;q)_\infty}    \;  {}_{1}\Phi_1\left(\begin{array}{c}\alpha
 \\
0
 \end{array}\Big|
q;vxt\right) \Bigg\}. \nonumber\ee
 
\end{proof}
For $y=0$ and $t=y$, the assertion  (\ref{ddredes}) is reduced to the Mehler's  formula for  $p_n(x,y,a)$  (\ref{mehl}). 
\section{Homogeneous  $q$-difference   operator $\widetilde{L}(a,b;\theta_{xy})$ and generalized Hahn polynomials}
In this section, we define  a  generalized Hahn polynomials $ h_n(x,y,a,b|q)$  in terms of homogeneous  $q$-shift operator  $\widetilde{L}(a,b; \theta_{xy})$. An extended generating function,  Mehler’s formula and   Rogers formula for the generalized Hahn polynomials $ h_n(x,y,a,b|q)$ are given. 

Let us recall the definition of the usual operator 

We now define the homogeneous Cauchy q-shift operator as follows.
\begin{definition}
The homogeneous  $q$-difference operator $\widetilde{L}(a,b;\theta_{xy})$ is defined  by
\be 
\label{mat}
\widetilde{L}(a,b; \theta_{xy})=\sum_{k=0}^\infty \frac{q^{({}^k_2)}\,(a;q)_k }{(q;q)_k}\, \left(b\,\theta_{xy}\right)^k.
\ee 
\end{definition}
 By means of $q$-hypergeometric series,   the operator (\ref{mat}) can be written as:
\be 
\widetilde{L}(a,b; \theta_{xy})= {}_{1}\Phi_1\left(\begin{array}{c}a
 \\
0
 \end{array}\Big|
q;-b\theta_{xy}\right).
\ee 
\begin{theorem}
\label{lasted}
The following operational formula holds true:
\be 
\label{asrule2}
\widetilde{L}(a,b; \theta_{xy})\Bigg\{\frac{(xt;q)_\infty}{(yt;q)_\infty}\Bigg\} = \frac{(xt;q)_\infty}{(yt;q)_\infty}\,{}_{1}\Phi_1\left(\begin{array}{c}a
 \\
0
 \end{array}\Big|
q; bt\right).
\ee
\end{theorem}
\begin{proof}
Let $f_{q}(x,y;a,b;t)$ be the function defined as
\be 
f_{q}(x,y;a,b;t)=\widetilde{L}(a,b; \theta_{xy})\Bigg\{ \frac{(xt;q)_\infty}{(yt;q)_\infty}\Bigg\}.  \nonumber
\ee
Using (\ref{deule2}) and  (\ref{mat}), respectively, we have
\bea
 \label{ale1}
f_{q}(x,y;a,b;t)&=&\sum_{k=0}^\infty\frac{  q^{({}^k_2)}\,(a;q)_k\,b^k }{(q;q)_k}\,  \theta_{xy}^k\Bigg\{ \frac{(xt;q)_\infty}{(yt;q)_\infty}\Bigg\} 
\cr
&=&\frac{(xt;q)_\infty}{(yt;q)_\infty}\sum_{k=0}^\infty \frac{ (-1)^k q^{({}^k_2)}\,(a;q)_k }{(q;q)_k}\,(bt)^k=\frac{(xt;q)_\infty}{(yt;q)_\infty} \;  {}_{1}\Phi_1\left(\begin{array}{c}a
 \\
0
 \end{array}\Big|
q;bt\right), \nonumber
\eea 
which completes the  proof of assertion (\ref{asrule2}). 
\end{proof}
\begin{definition}
In terms of the q-shifted factorial and the Cauchy polynomials $p_n(x,y)$,  we write
\be
\label{defR}
h_n(x,y,a,b|q)=\sum_{k=0}^n  \qbinomial{n}{k}{q} (-1)^{k} q^{\binom{k}{2}}\,b^{k}\,(a;q)_{k}\, p_{n-k}(y,x).
\ee
\end{definition}
Remark that, for $a=0$ and $b=z$, the  $q$-polynomials $h_n(x,y,a,b|q)$ are the well known trivariate $q$-polynomials $F_n(x,y,z;q)$ investigated by Mohammed (see \cite{Mohammed} for more details), i.e. 
\be
h_n(x,y,0,z|q) = (-1)^{n} q^{\binom{n}{2}}F_n(x,y,z;q).
\ee
The  $q$-polynomials $h_n(a,b,x,y|q)$ is a general form of Hahn polynomials $\psi_n^{(a)}(x|q)$
according   when we are setting $a=0,\,b=1$ and $y = ax$. Also, if we let $a=0, y = ax$ and $b=y$, we get the second Hahn polynomials $\psi_n^{(a)}(x,y|q)$.  Hence we have the following special substitutions:
\be
h_n(x,ax,0,1|q) = (-1)^{n} q^{\binom{n}{2}}\psi_n^{(a)}(x|q)\mbox{ and }  h_n(x,ax,0,y|q) = (-1)^{n} q^{\binom{n}{2}}\psi_n^{(a)}(x,y|q).
\ee
The polynomials (\ref{defR}) can be represented by the homogeneous $q$-difference operator (\ref{mat}) as follows:
 \be
\label{sddefR}
\widetilde{L}(a,b; \theta_{xy})\Big\{p_n(y,x)\Big\}=h_n(x,y,a,b|q).
\ee
Indeed, 
 \bea
\widetilde{L}(a,b; \theta_{xy})\Big\{p_n(y,x)\Big\}&=& \sum_{k=0}^\infty \frac{ q^{({}^k_2)}\,(a;q)_k }{(q;q)_k}\, \left(b\,\theta_{xy}\right)^k\Big\{p_n(y,x)\Big\}\cr
&=&\sum_{k=0}^\infty \frac{ q^{({}^k_2)}\,(a;q)_k \, b^k}{(q;q)_k}\,(-1)^k\frac{ (q;q)_n }{(q;q)_{n-k}} p_{n-k}(y,x) =h_n(x,y,a,b|q).  \nonumber
\eea
\begin{theorem}(Generating function for $h_n(x,y,a,b|q)$)
\label{csorro}
The following generating function holds for the $q$-polynomials $h_n(a,b,x,y|q)$:
\be
\label{aaGene}
 \sum_{n=0}^\infty h_n(x,y,a,b,|q) \frac{t^n}{(q;q)_n}= \frac{(xt;q)_\infty}{(yt;q)_\infty} \;  {}_{1}\Phi_1\left(\begin{array}{c}a
 \\
0
 \end{array}\Big|
q;bt\right).
\ee
\end{theorem}
\begin{proof} We observe that:
\be 
\sum_{n=0}^\infty  h_n(x,y,a,b|q) \frac{t^n}{(q;q)_n}=\sum_{n=0}^\infty \widetilde{L}(a,b; \theta_{xy})\Big\{p_n(y,x)\Big\}\frac{t^n}{(q;q)_n} 
\nonumber
\ee 
\be
= \widetilde{L}(a,b; \theta_{xy})\Bigg\{\sum_{n=0}^\infty p_n(y,x)\frac{t^n}{(q;q)_n}\Bigg\} =\widetilde{L}(a,b;\theta_{xy})\Bigg\{ \frac{(xt;q)_\infty}{(yt;q)_\infty}\Bigg\}
\nonumber
\ee 
\be
= \frac{(xt;q)_\infty}{(yt;q)_\infty} \;  {}_{1}\Phi_1\left(\begin{array}{c}a
 \\
0
 \end{array}\Big|
q;bt\right)\nonumber
\ee 
which evidently completes the proof of assertion (\ref{aaGene}). 
\end{proof}

\begin{theorem}(Extended generating function for $h_n(x,y,a,b|q)$)
\label{csorro}
We have:
\be
\label{aqaGene}
 \sum_{n=0}^\infty h_{n+k}(x,y,a,b|q) \frac{t^n}{(q;q)_n}=\widetilde{L}(a,b; \theta_{xy})\Bigg\{\frac{p_k(y,x) (xt;q)_\infty}{ (xt;q)_k(yt;q)_\infty}\Bigg\}.
\ee
\end{theorem}
\begin{proof} We observe that:
\bea 
\sum_{n=0}^\infty  h_{n+k}(x,y,a,b|q) \frac{t^n}{(q;q)_n}&=&\sum_{n=0}^\infty \widetilde{L}(a,b; \theta_{xy})\Big\{p_{n+k}(y,x)\Big\}\frac{t^n}{(q;q)_n}\cr
&=& \widetilde{L}(a,b; \theta_{xy})\Bigg\{p_k(y,x)\sum_{n=0}^\infty p_n(y,q^kx)\frac{t^n}{(q;q)_n}\Bigg\} \cr
&=&\widetilde{L}(a,b;\theta_{xy})\Bigg\{p_k(y,x)\frac{(xtq^k;q)_\infty}{(yt;q)_\infty}\Bigg\}\cr
&=&\widetilde{L}(a,b; \theta_{xy})\Bigg\{\frac{p_k(y,x) (xt;q)_\infty}{ (xt;q)_k(yt;q)_\infty}\Bigg\}\nonumber
\eea
which evidently completes the proof of assertion (\ref{aaGene}). 
\end{proof}
Setting $k=0$ in  (\ref{aqaGene}), we get the generating function (\ref{aaGene}) for the $q$-polynomials $h_n(x,y,a,b|q)$.  

\begin{theorem}(Roger's-Type formula for $h_n(x,y,a,b|q)$)
\label{Rog}
The following Rogers-Type formula  holds for the $q$-polynomials $h_n(x,y,a,b|q)$:
\be  
\label{rule2}
\sum_{n=0}^\infty\sum_{m=0}^\infty h_{n+m}(x,y,a,b|q)    \frac{  t^n}{(q;q)_n}  \frac{  s^m}{(q;q)_m}=  \widetilde{L}(a,b; \theta_{xy})\Bigg\{\frac{(xs;q)_\infty}{(ys;q)_\infty} \;  {}_{2}\Phi_1\left(\begin{array}{c}x/y,0
 \\
xs
 \end{array}\Big|
q;yt\right)\Bigg\}
\ee
\be
=  \widetilde{L}(a,b; \theta_{xy})\Bigg\{\frac{(xs;q)_\infty}{(yt,ys;q)_\infty} \;  {}_{1}\Phi_1\left(\begin{array}{c}ys
 \\
xs
 \end{array}\Big|
q;xt\right)\Bigg\}.\label{rule22}
\ee 
\end{theorem}
 \begin{proof}
\be 
 \sum_{n=0}^\infty\sum_{m=0}^\infty h_{n+m}(x,y,a,b|q)    \frac{  t^n}{(q;q)_n}  \frac{  s^m}{(q;q)_m}
=  
\sum_{n=0}^\infty\sum_{m=0}^\infty \widetilde{L}(a,b; \theta_{xy})\Big\{p_{n+m}(y,x)\Big\}\frac{t^n}{(q;q)_n}  \frac{  s^m}{(q;q)_m}\nonumber \ee
\be 
= \widetilde{L}(a,b; \theta_{xy})\Bigg\{\sum_{n=0}^\infty p_n(y,x)\frac{t^n}{(q;q)_n}\sum_{m=0}^\infty p_m(y,q^nx)\frac{s^m}{(q;q)_m}\Bigg\} \nonumber\ee
\be
=\widetilde{L}(a,b; \theta_{xy})\Bigg\{\sum_{n=0}^\infty p_n(y,x)\frac{t^n}{(q;q)_n} \frac{(xsq^n;q)_\infty}{(ys;q)_\infty}\Bigg\}\nonumber\ee
\be
=\widetilde{L}(a,b; \theta_{xy})\Bigg\{ \frac{(xs;q)_\infty}{(ys;q)_\infty}\sum_{n=0}^\infty p_n(y,x)\frac{t^n}{(xs;q)_n(q;q)_n} \Bigg\}\nonumber\ee
\be
=\widetilde{L}(a,b; \theta_{xy})\Bigg\{ \frac{(xs;q)_\infty}{(ys;q)_\infty}\sum_{n=0}^\infty  \frac{(x/y;q)_n\,(yt)^n}{(xs;q)_n(q;q)_n} \Bigg\}\nonumber\ee
\be
= \widetilde{L}(a,b;\theta_{xy})\Bigg\{\frac{(xs;q)_\infty}{(ys;q)_\infty} \;  {}_{2}\Phi_1\left(\begin{array}{c}x/y,0
 \\
xs
 \end{array}\Big|
q;yt\right)\Bigg\}.\nonumber
\ee 
 The proof of (\ref{rule22}) of Theorem \ref{Rog}  will be completed when we replace $y$ by $x$ and $x$ by $y$, respectively,  in  (\ref{sub}).  
\end{proof}

\begin{theorem}(Mehler's  formula for $h_n(x,y,a,b|q)$)
We have:
\be 
\label{rule2}
\sum_{n=0}^\infty h_n(x,y,a,b|q) h_{n}(u,v,c,d|q)    \frac{  t^n}{(q;q)_n}  
\ee 
\be 
=\widetilde{L}(a,b; \theta_{xy})\Bigg\{ \frac{(xt;q)_\infty}{(yt ;q)_\infty}  {}_{3}\Phi_3\left(\begin{array}{c} x/y, u/v, c
 \\
0,0,xt
 \end{array}\Big|
q; dvyt\right)  \Bigg\}. \nonumber\ee
\end{theorem}
 \begin{proof}
\be 
 \sum_{n=0}^\infty h_n(x,y,a,b|q) h_{n}(u,v,c,d|q)    \frac{  t^n}{(q;q)_n}
=  
\sum_{n=0}^\infty  \widetilde{L}(a,b; \theta_{xy})\Big\{p_{n }(y,x)\Big\} h_{n}(u,v,c,d|q) \frac{t^n}{(q;q)_n}  \nonumber \ee
\be 
= \widetilde{L}(a,b; \theta_{xy})\Bigg\{\sum_{n=0}^\infty p_n(y,x)\frac{t^n}{(q;q)_n}\sum_{k=0}^n  \qbinomial{n}{k}{q}(-1)^{k}q^{\binom{k}{2}}\,d^{k}\,(c;q)_{k}\, p_{n-k}(v,u)\Bigg\} \nonumber\ee
\be 
= \widetilde{L}(a,b; \theta_{xy})\Bigg\{\sum_{n=0}^\infty\sum_{k=n}^\infty \frac{(-1)^{k}   q^{\binom{k}{2}}\,d^{k}\,t^n}{(q;q)_k(q;q)_{n-k}}(c;q)_{k}\, p_{n-k}(v,u) p_{n}(y,x)\Bigg\} \nonumber\ee
\be 
= \widetilde{L}(a,b;\theta_{xy})\Bigg\{\sum_{n=0}^\infty\sum_{k=0}^\infty \frac{ (-1)^kq^{\binom{k}{2}}\,d^{k} t^{n+k}}{(q;q)_k(q;q)_{n}} \,(c;q)_{k}\, p_k(v,u) p_{n+k}(y,x)\Bigg\} \nonumber\ee
\be 
= \widetilde{L}(a,b; \theta_{xy})\Bigg\{\sum_{k=0}^\infty p_k(v,u) p_{k} (y,x)  \frac{ (-1)^k q^{\binom{k}{2}}  (c;q)_{k}\,(dt)^{k}}{(q;q)_{k}}     \sum_{n=0}^\infty p_{n} (y,q^kx) \frac{  t^{n}}{(q;q)_{n}}   \,\Bigg\} \nonumber\ee
\be 
= \widetilde{L}(a,b; \theta_{xy})\Bigg\{\sum_{k=0}^\infty p_k(v,u) p_{k} (y,x)  \frac{ (-1)^k q^{\binom{k}{2}}  (c;q)_{k}\,(dt)^{k}}{(q;q)_{k}}  \frac{(xtq^k;q)_\infty}{(yt ;q)_\infty}     \Bigg\} \nonumber\ee
\be 
= \widetilde{L}(a,b; \theta_{xy})\Bigg\{ \frac{(xt;q)_\infty}{(yt ;q)_\infty}    \sum_{k=0}^\infty      \frac{(-1)^k q^{\binom{k}{2}} (x/y,u/v,c;q)_{k} }{ (xt ;q)_{k}}   \frac{ (dvyt)^{k}}{(q;q)_{k}}   \Bigg\} \nonumber\ee
\be 
= \widetilde{L}(a,b; \theta_{xy})\Bigg\{ \frac{(xt;q)_\infty}{(yt ;q)_\infty}  {}_{3}\Phi_3\left(\begin{array}{c} x/y, u/v, c
 \\
0,0,xt
 \end{array}\Big|
q; dvyt\right)  \Bigg\}.\nonumber\ee

\end{proof}
 


\begin{thebibliography}{}
\bibitem{Mohammed}M. A. Abdlhusein, Two operator representations for the trivariate q-polynomials and Hahn polynomials, Ramanujan J. {\bf 40},   491-509 (2016).

\bibitem{Andrews}G. E.  Andrews,    The theory of partitions, Cambridge Univ. Press, (1985).

\bibitem{Chen2003} W. Y. C. Chen,  A. M. Fu  and  B. Zhang,   The homogeneous $q$-difference operator. Adv. App. Math. {\bf 31}    659-668 (2003).

\bibitem{Liu97} W. Y. C. Chen and Z.-G. Liu,  Parameter augmenting for basic hypergeometric series, II, J. Combin. Theory, Ser. A, {\bf 80}  pp. 175-195, (1997).

\bibitem{GasparRahman}G.   Gasper  and M.  Rahman,  Basic Hypergeometric Series, 2nd edn. Cambridge University Press, Cambridge (2004). 

\bibitem{Goldman}J. Golman and G.-C. Rota, {\it On the foundations of combinatorial theory, IV: Finite vector spaces ans Eulerien generating functions}, Sut. Appl. Math. {\bf 49}, 239-258 (1970).

\bibitem{Johnson} W. P. Johnson,   $q$-Extensions of identities of Abel-Rothe type Discrete Math. {\bf 159} 161-77  (1995).

\bibitem{Ismail1981}E. C.  Ihrig   and M. E. H.  Ismail,   A $q$-umbral calculus, J. Math. Anal. Appl. {\bf 84}, 178-207 (1981).

\bibitem{Koekock}R. Koekock and R. F. Swarttouw, {\it The Askey-scheme of hypergeometric orthogonal polynomials and its $q$-analogue report}, Delft University of Technology,
(1998).

\bibitem{Roman1982} S. Roman,   The theory of the umbral calculus I.  J. Math. Anal. Appl. {\bf 87},  58-115 (1982).

\bibitem{Roman1985}S. Roman, More on the umbral calculus, with emphasis on the 
$q$-umbral calculus,  J. Math. Anal. Appl. {\bf 107} 222-54 (1985 ).

\bibitem{Saadsukhi}H. L. Saad   and  A. A. Sukhi, Another homogeneous $q$-difference operator. Appl. Math. Comput. {\bf 215} 4332-4339 (2010).

\bibitem{Saad} H. L. Saad and A. A. Sukhi,  The $q$-Exponential Operator, Appl. Math. Sci. {\bf 7}, 6369-6380 (2005).

\bibitem{SLATER} L. J. Slatter, {\it Generalized Hypergeometric Functions}, Cambridge Univ. Press, Cambridge/London/New York, (1966).

\bibitem{SrivastaAbdlhusein}H. M. Srivastava and M. A. Abdlhusein, New forms of the Cauchy operator and some of their \break applications,  Russian J. Math. Phys. {\bf 23}, 124--134 (2016).

\bibitem{SrivastaKarlsson}  H. M. Srivastava and 
P. W. Karlsson,  {\it Multiple Gaussian Hypergeometric Series}, Halsted (Ellis Horwood, Chichester); Wiley, New York, (1985). 

\end{thebibliography}
\end{document}